\documentclass[a4paper,12pt,reqno]{amsart}
\usepackage{amsfonts}
\usepackage{amsmath}
\usepackage{amssymb}
\usepackage[a4paper]{geometry}
\usepackage{mathrsfs}
\usepackage{csquotes}
\usepackage{enumitem}
\usepackage[colorlinks]{hyperref}
\renewcommand\eqref[1]{(\ref{#1})} 
\numberwithin{equation}{section}
\theoremstyle{plain}
\newtheorem{thm}{Theorem}[section]

\newtheorem{cor}[thm]{Corollary}
\newtheorem{lem}[thm]{Lemma}
 
\theoremstyle{definition}

\newcommand{\Rn}{\mathbb R^{n}}

\begin{document}

   \title[Trace formulae for degenerate parabolic equations]
   {Trace formulae of potentials for degenerate parabolic equations}

\author[M. Karazym]{Mukhtar Karazym}
\address{
  Mukhtar Karazym:
  \endgraf
  Department of Mathematics
  \endgraf
Nazarbayev University, Kazakhstan
  \endgraf
  {\it E-mail address} {\rm mukhtar.karazym@nu.edu.kz}
  }

 \author[D. Suragan]{Durvudkhan Suragan}
\address{
	Durvudkhan Suragan:
	\endgraf
	Department of Mathematics
	\endgraf
Nazarbayev University, Kazakhstan
	\endgraf
	{\it E-mail address} {\rm durvudkhan.suragan@nu.edu.kz}
} 

\thanks{The authors were supported
		by the Nazarbayev University program 091019CRP2120. No new data was collected or generated during the course of research.}

     \keywords{degenerate parabolic equation, Poisson integral, layer potential, Fourier transform}
     \subjclass[2010]{47G40, 35K65}

     \begin{abstract} In this paper, we analyze main properties of the volume and layer potentials as well as the Poisson integral for a  multi-dimensional degenerate parabolic equation. As consequences, we obtain trace formulae of the heat volume potential and the Poisson integral which solve Kac's problem for degenerate parabolic equations in cylindrical domains.
      \end{abstract}
     \maketitle

     \tableofcontents

\section{Introduction}
The layer potential method (or potential theory) for parabolic equations has a long history (see, e.g. \cite{FR}) and it has been intensively applied to solve initial and initial-boundary value problems of  parabolic partial differential equations throughout the last decades. To construct the method, elements of the potential theory, namely, the (heat) volume potential/Poisson integral , the single layer potential and the double layer potential play a key role. Although many of the basic ideas of the potential theory already exist and are intensively being studied, still specific (nonclassical) partial differential equations are required for their development and new approaches.

In \cite{M1}, the author studied the following one-dimensional  degenerate-type parabolic equation in the semi-infinite domain
\begin{equation}\label{malyshev1}
\frac{\partial u(x,t)}{\partial t}-a(t)\frac{\partial^{2} u(x,t)}{\partial x^{2}}=f(x,t),\quad x>0,~t>0,
\end{equation} 
where $f(x,t)$ is bounded in  the strip $\mathbb{R}\times[0,T]$, $0<T<\infty$. 

Here the coefficient $a(t)$ satisfies one of the following two assumptions: 
\begin{itemize}
\item[\emph{i.}] $a(t)$ is nonnegative and becomes zero only at isolated points;

\item[\emph{ii.}] A function $a_{1}(t)$ defined by
\begin{equation*}
a_{1}(t):=\int\limits_{0}^{t} a(z)\,dz
\end{equation*}
is positive for all $t>0$, allowing $a(t)$ to be negative in an interval. 
\end{itemize}

In particular, the author obtained solutions of the initial boundary value problems for equation \eqref{malyshev1} by using the potential theory.

The goal of the present paper is to construct the potential theory for the multi-dimensional version of the degenerate parabolic equation \eqref{malyshev1} and to analyse its consequences. To achieve this aim, first by using the Fourier transform we find the fundamental solution of the  multi-dimensional degenerate parabolic equation in an explicit form. Then we develop ``degenerate" potential theory, which is based on a use of the explicit representation of the fundamental solution for analysing, in this setting, a complete parallel of the classical heat potential and regularity theory. Note that our ideas are also closely related to the recent development on the potential theory of hypoelliptic differential equations (see \cite[Chapter 11]{RS1}). 

Thus, in this paper, we present ``degenerate" versions of the volume (heat) potential, the Poisson integral, the double and single layer potentials. In addition, their main properties will be discussed in details. As consequences, we consider Cauchy problems and initial-boundary value problems in cylindrical domains.

Moreover, we are also interested in the question that what boundary condition can be put on the ``degenerate'' volume potential (and Poisson integral) on the lateral boundary of the cylindrical domain so that the degenerate parabolic equation with this boundary condition would have a unique solution in the cylindrical domain, which is still given by the same formula of the ``degenerate'' volume potential (and Poisson integral, correspondingly). In turn, it allows finding the trace of the ``degenerate'' volume potential (and Poisson integral) to the lateral boundary of the cylindrical domain. So, in the present paper,  boundary conditions for the ``degenerate'' volume potential and Poisson integral are established. The obtained boundary conditions are nonlocal in the space variables.  In the one-dimensional case, this problem was studied in \cite{SO}. The multi-dimensional version gives a new insight, that is, the constructed new (nonlocal) initial-boundary value problem can serve as an example of an explicitly solvable initial-boundary value problem in any cylindirical domain (with a smooth lateral surface) for the degenerate parabolic equation. 

Note that the origin of the question goes back to M. Kac's lecture \cite{Kac} (cf. \cite{KS1} and \cite{KS2}). Therefore, the analogical questions for the elliptic and hypoelliptic cases are called Kac's problems. For discussions in this direction, we refer \cite[Chapter 11]{RS1}
as well as references therein.

The brief outline of the paper is as follows: in Section \ref{sec2}, we discuss Cauchy problems for the multi-dimensional degenerate parabolic equation and find its fundamental solution explicitly. We prove the existence and uniqueness theorems for the Cauchy problems.  In Section \ref{sec3}, we analyse layer potentials, in particular, we obtain continuity results and jump relations. Finally, in Section \ref{sec4}, we present trace formulae for the volume potential and Poisson integral.

\section{Fundamental solution and Cauchy problems}\label{sec2}
We consider the degenerate parabolic equation 
\begin{equation}\label{e1.1}
\lozenge_{a}u(x,t):=\frac{\partial u(x,t)}{\partial t}-a(t)\Delta_{x}u(x,t) = f(x,t),
\end{equation}
posed in a cylindrical domain $(x,t)\in\Omega \times(0,T)$, $0<T<\infty$, where the domain $\Omega$ is bounded in $\mathbb{R}^n$, $n\geq2$, with Lyapunov boundary $\partial \Omega \in C^{1+\lambda},~0<\lambda<1$, $f$ is any given function.
Here and throughout this paper the coefficient $a(t)\in L_{1}[0,T]$ is defined in $[0,T]$ and satisfies one of the following two assumptions:

(a) $a(t)$ is nonnegative and becomes zero only at isolated points;

(b) A function $a_{1}(t)$ defined by
\begin{equation*}
a_{1}(t):=\int\limits_{0}^{t} a(z)dz 
\end{equation*}
is positive for all $t>0$, allowing $a(t)$ to be negative in an interval. 

In our computations, we also use a function $b(t,\tau)$ defined by the formula
\begin{equation*}
b(t,\tau):=\int\limits_\tau^t a(z)dz=a_{1}(t)-a_{1}(\tau), \quad (b(t,0)=a_{1}(t)).
\end{equation*}
Note that if $a(t)$ satisfies the assumption (a), then $b(t,\tau)$ is positive for all $t>\tau>0$.

First of all, we present the fundamental solution of equation \eqref{e1.1} by using the Fourier transform in an explicit form. 
\begin{lem} \label{lem2.1}
	Under the assumption (b) the fundamental solution of equation \eqref{e1.1} can be represented as
	\begin{equation} \label{e2.1}
	\varepsilon_{n, b}( x,t):=\varepsilon_{n}\big( x,a_{1}(t)\big)=\frac{\theta(t) e^{-\frac{|x|^{2}}{4a_{1}(t)}}}{\big(4\pi a_{1}(t)\big)^{\frac{n}{2}}},\quad (x,t)\in\Rn\times\mathbb{R},
	\end{equation}
	where $\varepsilon_{n}$ is the fundamental solution of the standard heat  operator, $\theta$ is the Heaviside function and $|x|=\sqrt{x_{1}^{2}+...+x_{n}^{2}}$ is the usual Euclidean norm.
	\end{lem}
	\begin{proof}[Proof of Lemma \ref{lem2.1}] Consider the equation
		\begin{equation}\label{e2.2}
		\frac{\partial \varepsilon(x,t)}{\partial t}-a (t)\Delta_{x}\varepsilon(x,t)=\delta(x)\delta(t),\quad (x,t)\in\mathbb{R}^{n}\times\mathbb{R},
		\end{equation}	
		where $\delta$ is the Dirac distribution.
		Under the assumption (b) the fundamental solution of equation \eqref{e1.1} can be explicitly found by using the Fourier
		transform. So, applying the Fourier transform $F_{x}$ to  equation \eqref{e2.2}, we obtain that
		
		\begin{equation} \label{e2.3}
		\frac{\partial \widetilde{\varepsilon}(\xi,t)}{\partial t}+a(t)|\xi|^{2} \widetilde{\varepsilon}(\xi,t)=1(\xi)\delta(t),\quad (\xi,t)\in\mathbb{R}^{n}\times\mathbb{R},
		\end{equation}
		where
		  $$\widetilde{\varepsilon}(\xi,t)=F_{x}[\varepsilon](\xi,t)=\int_{\mathbb{R}^{n}} \varepsilon(x,t)e^{i\langle \xi,x \rangle}dx,\quad (i^2=-1),$$
		$1(\xi)$ is the identity function in $\Rn$ and the inner product in $\mathbb{R}^{n}$ is denoted by $\langle \cdot, \cdot\rangle$. The solution of  equation \eqref{e2.3} is 
		\begin{equation*}
		\widetilde{\varepsilon}(\xi,t)=\theta(t)e^{-|\xi|^{2}a_{1}(t)},\quad (\xi,t)\in\mathbb{R}^{n}\times\mathbb{R}.
		\end{equation*}
		Applying the inverse Fourier transform and its properties to the solution of equation \eqref{e2.3}, we obtain  \eqref{e2.1}. This completes the proof.
	\end{proof}

Note that the assumption (a) is the special case of the assumption (b).

With  substitution of the variables $\xi_{i}=\frac{x_{i}}{2\sqrt{a_{1}(t)}},~i=1,...,n$,
we have 
\begin{equation}\label{pr1}
\begin{aligned}
&\int_{\mathbb{R}^{n}}\varepsilon_{n, b}(x,t)dx=\frac{1}{\big(4\pi a_{1}(t)\big)^{\frac{n}{2}}}\int_{\mathbb{R}^{n}}e^{-\frac{|x|^{2}}{4a_{1}(t)}}dx\\
&=\prod_{i=1}^{n}\frac{1}{\pi^{\frac{n}{2}}}\int_{-\infty}^{\infty}e^{-\xi_{i}^{2}}d\xi_{i}=1, \quad t>0.
\end{aligned}
\end{equation}

Moreover, the fundamental solution $\varepsilon_{n, b}(x,t)$ has the property

\begin{equation}\label{pr2}
\varepsilon_{n, b}(x,t)\to \delta(x) \quad \mbox{ with }~ t \to 0+,
\end{equation}
for all $x\in \mathbb{R}^{n}$. 

Let us show \eqref{pr2}. Let $\psi$ be an infinitely many times differentiable function in $\mathbb{R}^{n}$ with compact support.
Then by using the polarization formula 
\begin{equation*}
\int_{\mathbb{R}^{n}}\widetilde{f}(|x|)dx=\omega_{n}\int_{0}^{\infty}\widetilde{f}(r)r^{n-1}dr,
\end{equation*}
where $\widetilde{f}$ is any integrable function in $\Rn$, $\omega_{n}=\frac{2\pi^{\frac{n}{2}}}{\Gamma\big(\frac{n}{2}\big)}$ and by using the mean value theorem, we obtain 
\begin{align*}
&\bigg|\int_{\mathbb{R}^{n}}\varepsilon_{n, b}(x,t)\big(\psi(x)-\psi(0)\big)dx\bigg|\leq \frac{A}{\big(4\pi a_{1}(t)\big)^{\frac{n}{2}}}\int_{\mathbb{R}^{n}}e^{-\frac{|x|^{2}}{4a_{1}(t)}}|x|dx\\
&=\frac{ A \omega_{n}}{\big(4\pi a_{1}(t)\big)^{\frac{n}{2}}}\int_{0}^{\infty}e^{-\frac{r^{2}}{4a_{1}(t)}}r^{n}dr=\frac{2A \omega_{n} \sqrt{a_{1}(t)}}{\pi^{\frac{n}{2}}}\int_{0}^{\infty}e^{-u^2}u^{n}du\\
&=2A\sqrt{a_{1}(t)},
\end{align*}
where $A$ is a positive constant.
Since  the function $a_{1}(t)$ is  continuous and nonnegative in $[0,T]$, by virtue of \eqref{pr1} , we obtain \eqref{pr2}, that is,
\begin{align*}
&\big(\varepsilon_{n, b}(x,t),\psi(x)\big):=\int_{\mathbb{R}^{n}}\varepsilon_{n, b}(x,t)\psi(x)dx=\psi(0)\int_{\mathbb{R}^{n}}\varepsilon_{n, b}(x,t)dx\\
&+\int_{\mathbb{R}^{n}}\varepsilon_{n, b}(x,t)\big(\psi(x)-\psi(0)\big)dx \to \big(\delta(x),\psi(x)\big):=\psi(0),~\text{as}~t \to 0+.
\end{align*}

Let $\varepsilon_{n, a}(x-\xi,b(t,\tau)):=\varepsilon_{n}(x-\xi,a_{1}(t)-a_{1}(\tau))$.

Since under the assumption (a) we have $b(t,\tau)>0$ for all $t>\tau>0$, it is easy to check that 

\begin{equation}\label{pr3}
\int_{\mathbb{R}^{n}}\varepsilon_{n, a}(x-\xi,b(t,\tau))d\xi=1, \quad t>\tau>0,~x\in\Rn.
\end{equation}

The degenerate parabolic potential defined by
\begin{equation}\label{volume}
(Vf)(x,t):=\int_{0}^{t}\int_{\Omega}
\varepsilon_{n, a}(x-\xi,b(t,\tau))f(\xi,\tau)d\xi d\tau,\quad x\in\Omega,~0<t<T,
\end{equation}
is called the volume potential, where $f$ is bounded in $\Omega\times[0,T]$ with $\operatorname{supp}f(\cdot,t)\subset \Omega$ for all $t\in[0,T]$.
\begin{thm} \label{thm2.2}
	Let $a(t)$ satisfy the assumption (a) and $f$ be a bounded function in the strip $\Omega\times[0,T]$ with $\operatorname{supp}f(\cdot,t)\subset \Omega$ for all $t\in[0,T]$. Then the volume potential with the density $f$ \eqref{volume}  admits the estimate
	\begin{equation}\label{e2.6}
	|(Vf)(x,t)|\leq t\sup_{(\xi,\tau)\in \Omega\times [0,t]}|f(\xi,\tau)|,\quad x\in\Omega,~0<t<T,
	\end{equation}
	 and solves  equation \eqref{e1.1} with the zero initial condition
	\begin{equation}\label{i2.8}
	u(\cdot,t)\to 0~\text{as}~ t \to  0+,\quad \text{in}~\Omega.
	\end{equation} 
\end{thm} 
\begin{proof}[Proof of Theorem \ref{thm2.2}]
	Since $\operatorname{supp}f(\cdot,t)\subset\Omega$ for all $0\leq t \leq T$, it is obvious that
	\begin{eqnarray*}
	\begin{aligned}
	&(Vf)(x,t)=\int_{0}^{t}\int_{\Omega}
	\varepsilon_{n, a}(x-\xi,b(t,\tau))f(\xi,\tau)d\xi d\tau\\
	&=\int_{0}^{t}\int_{\Rn}
	\varepsilon_{n, a}(x-\xi,b(t,\tau))f(\xi,\tau)d\xi d\tau,\quad x\in\Omega,~t\in (0,T).
	\end{aligned}
	\end{eqnarray*}
	Thus, by virtue of \eqref{pr3}, we obtain \eqref{e2.6}
	
	\begin{eqnarray*}
		\begin{aligned}
		&|(Vf)(x,t)|\leq \sup_{(\xi,\tau)\in \Rn\times [0,t]}|f(\xi,\tau)|\int_{0}^{t}\int_{\Rn}\varepsilon_{n, a}(x-\xi,b(t,\tau))d\xi d\tau\\
		&=t\sup_{(\xi,\tau)\in \Omega\times [0,t]}|f(\xi,\tau)|,\quad (x,t)\in\Omega\times(0,T).
		\end{aligned}
	\end{eqnarray*}
	 A direct calculation gives that  the volume potential $Vf$ satisfies equation \eqref{e1.1}.
	Also, we observe that  estimate \eqref{e2.6}  ensures convergence of \eqref{i2.8}.
\end{proof}
	The degenerate parabolic potential defined by
	\begin{equation}\label{Poisson}
	(P\varphi)(x,t):=\int_{\Omega} \varepsilon_{n, b}(x-\xi,t)\varphi(\xi) d\xi,\quad x\in\Omega,~0<t<T,
	\end{equation}
	is called the Poisson potential (see, e.g. [\cite{L1}, p. 153]), where $\varphi$ is a bounded function in $\mathbb{R}^{n}$ with  $\operatorname{supp}\varphi\subset\Omega$ and  $\varepsilon_{n, b}(x-\xi,t)=\varepsilon_{n}( x-\xi,a_{1}(t))$.
	
    \begin{thm} \label{thm2.3}	Let $a(t)$ satisfy the assumption (b). Let $\varphi$ be a bounded function in $\mathbb{R}^{n}$ with  $\operatorname{supp}\varphi\subset\Omega$. Then the Poisson  integral \eqref{Poisson} admits  the estimate
	\begin{equation}\label{e2.7}
	|(P\varphi)(x,t)|\leq \sup_{\xi \in \Omega}|\varphi(\xi)|,\quad x\in\Omega,~0<t<T,
	\end{equation}
	and solves the equation
	\begin{equation}\label{i2.12}
	\lozenge_{a}u=0,\quad\text{in}~\Omega\times(0,T).
	\end{equation} 
	Moreover, if $\varphi$ is a continuous bounded function in $\Rn$ with $\operatorname{supp}\varphi\subset\Omega$, then the Poisson integral $P\varphi$ belongs to the class $C^{\infty}$ and satisfies the initial condition 	
	\begin{equation}\label{i2.11}
	u(\cdot,0)=\varphi,\quad\text{in}~\Omega,
	\end{equation}
	providing its continuous extension to $\Omega\times[0,T)$.
\end{thm}
\begin{proof}[Proof of Theorem \ref{thm2.3}]
	Since $\operatorname{supp}\varphi\subset\Omega$, it is obvious that
	\begin{eqnarray*}
		\begin{aligned}
		&(P\varphi)(x,t)=\int_{\Omega}\varphi(\xi)
		\varepsilon_{n, b}(x-\xi,t)d\xi\\
		&=\int_{\mathbb{R}^n}\varphi(\xi)
		\varepsilon_{n, b}(x-\xi,t)d\xi, \quad(x,t)\in\Omega\times(0,T).
		\end{aligned}
	\end{eqnarray*}
	For $x\in\Omega$ and $0<t<T$,   we have the estimate
	\begin{equation*}
	|(P\varphi)(x,t)|\leq \sup_{\xi \in \mathbb{R}^n}|\varphi(\xi)|\int_{\mathbb{R}^{n}}\varepsilon_{n, b}(x-\xi,t)d\xi=\sup_{\xi \in \Omega}|\varphi(\xi)|.
	\end{equation*}
	Since for all $x\in \Omega$ and $t\in (0,T)$ differentiation and integration can be interchanged in  \eqref{Poisson}, it is straightforward to check that $P\varphi$ satisfies \eqref{i2.12}.
	
	Let $\varphi$ be a continuous bounded function in $\mathbb{R}^{n}$ with  $\operatorname{supp}\varphi\subset\Omega$. Taking  into account \eqref{pr2}, we see that $P\varphi$ satisfies  initial condition \eqref{i2.11}. Now we substitute $\xi=x+2\sqrt{a_{1}(t)}z$ to obtain
	\begin{equation*}
	P(\varphi)(x,t)=\frac{1}{\pi^{\frac{n}{2}}}\int_{\mathbb{R}^n}
	\varphi(x+2\sqrt{a_{1}(t)}z)e^{-|z|^{2}}dz.
	\end{equation*}
	The assumption for $\varphi$ provides its boundedness and uniformly continuity. Let $M_{\varphi}>0$ be an upper bound for $\varphi$. Since $\varphi$ is a uniformly continuous function, for any $\varepsilon>0$, there exists $\delta>0$ such that $|\varphi(x)-\varphi(\xi)|<\frac{\varepsilon}{2}$ for all $x,\xi \in \mathbb{R}^{n}$ with $|x-\xi|<\delta$.
	Then for any $\varepsilon>0$ we can choose $r>0$ such that
	$$\frac{1}{\pi^{\frac{n}{2}}}\int_{|z|\geq r}e^{-|z|^{2}}dz\leq \frac{\varepsilon}{4M_{\varphi}}.$$
	Since $a_{1}(t)$ is a continuous function in $[0,T]$,  for any $\eta>0$ there exists $\delta_{\eta}>0$ such that $|a_{1}(t)|<\eta$ for all $t\in [0,T]$ with $t<\delta_{\eta}$. Setting $\eta=\frac{\delta^{2}}{4r^{2}}$ and using the fact that for $|z|\leq r$ and $t<\delta_{\eta}$ we have $2\sqrt{a_{1}(t)}z<2\sqrt{\eta}r=\delta$, we deduce that
	\begin{align*}
	&\Bigg|	\frac{1}{\big(4\pi a_{1}(t)\big)^{\frac{n}{2}}}\int_{\mathbb{R}^n}e^{-\frac{|x-\xi|^{2}}{4a_{1}(t)}}\varphi(\xi)d\xi-\varphi(x)\Bigg|\\
	&=\Bigg|\frac{1}{\pi^{\frac{n}{2}}}\int_{\mathbb{R}^n}\Big(
	\varphi\big(x+2\sqrt{a_{1}(t)}z\big)-\varphi(x)\Big)e^{-|z|^{2}}dz\Bigg|\\
	&<\frac{\varepsilon}{2\pi^{\frac{n}{2}}}\int_{|z|\leq r}e^{-|z|^{2}}dz+\frac{2M_{\varphi}}{\pi^{\frac{n}{2}}}\int_{|z|\geq r}e^{-|z|^{2}}dz<\varepsilon,
	\end{align*}
	for all $x \in \mathbb{R}^{n}$ and $t<\delta_{\eta}$. This implies continuity of the potential $P(\varphi)$ at $t=0$ and $P(\varphi)(\cdot,0)=\varphi$ in $\Omega$. 
\end{proof}

\section{Layer potentials}\label{sec3}

Let $a(t)$ satisfy the assumption (a) and  $\varphi \in C(\partial \Omega \times [0,T])$. Then the single layer potential for the degenerate parabolic equation \eqref{e1.1}  can be defined  by 

\begin{equation}\label{singlelayer}
(S\varphi)(x,t):=\int_{0}^{t}\int_{\partial\Omega}
\varepsilon_{n, a}(x-\xi,b(t,\tau))\varphi(\xi,\tau)a(\tau)dS_{\xi} d\tau,
\end{equation}
and the double layer potential can be defined  by
\begin{equation}\label{doublelayer}
(D\varphi)(x,t):=\int_{0}^{t}\int_{\partial\Omega}
\frac{\partial\varepsilon_{n, a}(x-\xi,b(t,\tau))}{\partial\nu(\xi) }\varphi(\xi,\tau)a(\tau)dS_{\xi} d\tau,
\end{equation}
where $\nu(\xi)$ is the outward unit normal at the boundary point $\xi \in \partial \Omega$.
If $x \in \partial\Omega$, then these integrals are improper and defined as $\lim\limits_{h\to 0}\int_{0}^{t-h}\int_{\partial\Omega}$.
\begin{thm}\label{thm3.1}
	The single layer potential with bounded measurable density $\varphi$ is continuous in $\mathbb{R}^{n}\times \mathbb{R}_{+}$. In particular, it is continuous across the boundary $\partial\Omega$.
\end{thm}
\begin{proof}[Proof of Theorem \ref{thm3.1}]
	If we prove that $\varepsilon_{n, a}(x-\xi,b(t,\tau))a(\tau)$ is locally integrable, the proof follows from \cite[p. 7, Lemma 1]{F4}. So, let us show $\varepsilon_{n, a}(x-\xi,b(t,\tau))a(\tau)$ is locally integrable.
	
	We have
	\begin{equation}\label{inq6}
	s^{\beta}e^{-s}\leq \beta^{\beta}e^{-\beta},
	\end{equation}
	for all $0<s,\beta<\infty$. Using \eqref{inq6} for the case $s=\frac{|x-\xi|^{2}}{4b(t,\tau)}$, $\beta=\frac{n}{2}-\gamma$, we have
	$$\big|\varepsilon_{n, a}(x-\xi,b(t,\tau))a(\tau)\big|\leq \frac{C|a(\tau)|}{|x-\xi|^{n-2\gamma}\big(b(t,\tau)\big)^{\gamma}},$$
	where $0<\gamma<\frac{n}{2}$. Hence, choosing $\frac{1}{2}<\gamma<1$, we  observe that $\varepsilon_{n, a}(x-\xi,b(t,\tau))a(\tau)$ is locally integrable. This completes the proof.
\end{proof}

A direct calculation gives that the double layer potential and single layer potential are infinitely many times differentiable solutions of \eqref{i2.12} in $\Omega\times(0,T)$.
Both layer potentials can be continuously extended to $\Omega\times [0,T) $  by setting $(D\varphi)(x,0)=0$ and $(S\varphi)(x,0)=0$ for all $x\in\Omega$.

We pay special attention to the boundary behaviour of the gradient of the single layer potential $(S\varphi)(x,t)$ when $\Omega \ni x\to x_{0}\in \partial\Omega$ along nontangential directions. For any  $x_{0}\in \partial\Omega$, we denote by $K=K(x_{0})$ a finite closed cone in $\mathbb{R}^{n}$ with vertex $x_{0}$ such that 
$K(x_{0})\subset\Omega\cup\{x_{0}\}$.
As in \cite{F4} and \cite{K5}, we prove the following theorem.
	\begin{thm}\label{thm3.2}
	Let $\partial\Omega\in C^{1+\lambda}$, $0<\lambda<1$. Let $\varphi$ be a continuous function on $\partial\Omega\times[0,T]$. Then, for any $x_{0}\in \partial\Omega$ and $t\in (0,T]$, the single layer potential \eqref{singlelayer} satisfies the jump relation
	\begin{equation}\label{e3.1}
	\begin{aligned}
	&\lim\limits_{\stackrel{x \to x_{0}}{x \in K}}\big\langle \nabla_{x}(S\varphi)(x,t),\nu(x_{0})\big\rangle=\frac{1}{2}\varphi(x_{0},t)\\
	&+\int_{0}^{t}\int_{\partial\Omega}
	\frac{\partial\varepsilon_{n, a}(x_{0}-\xi,b(t,\tau))}{\partial\nu(x_{0}) }\varphi(\xi,\tau)a(\tau)dS_{\xi} d\tau,
	\end{aligned}
	\end{equation}
	where the limit is taken along the outward normal $\nu(x_{0})$ and $\nabla_{x}$ is the usual gradient.
\end{thm}
\begin{proof}[Proof of Theorem \ref{thm3.2}] For convenience of a reader,  let us  rewrite the formula of the single layer potential
	\begin{equation*}
	(S\varphi)(x,t)=\int_{0}^{t}\int_{\partial\Omega}
	\varepsilon_{n, a}(x-\xi,b(t,\tau))\varphi(\xi,\tau)a(\tau)dS_{\xi} d\tau.
	\end{equation*}
	
	Let $T(x_{0})$  denote the tangent hyperplane to the boundary $\partial\Omega$ at the point $x_{0}$ and $\partial\Omega_{r}:=B(x_{0},r)\cap \partial\Omega $, where $B(x_{0},r)$ is the open ball of radius $r>0$ centred at the point $x_{0}$ in $\mathbb{R}^{n}$. Since $\partial\Omega \in C^{1+\lambda},~0<\lambda<1,$ if $r>0$ is small enough, the orthogonal projection $\Phi:\partial\Omega_{r}\to T(x_{0})$ is one-to-one map. We denote its image by $\partial\Omega_{r}^{'}:=\Phi(\partial\Omega_{r})$.
	For convenience we split the inner product in \eqref{e3.1} into two parts
	\begin{equation*}
	\big\langle\nabla_{x}(S\varphi)(x,t),\nu(x_{0})\big\rangle= I_{r}(x,t)+J_{r}(x,t),
	\end{equation*}
	with
	\begin{equation*}
	\big\langle\nabla_{x}\varepsilon_{n, a}(x-\xi,b(t,\tau)),\nu(x_{0})\big\rangle=-\frac{\langle x-\xi, \nu(x_{0})\rangle}{2^{n+1}\big[b(t,\tau)\big]^{\frac{n}{2}+1}\pi^{\frac{n}{2}}}e^{-\frac{|x-\xi|^2}{4b(t,\tau)}},
	\end{equation*}	
	and $I_{r}(x,t)$ is defined by
	\begin{equation}\label{idelta}
	I_{r}(x,t):=\int_{0}^{t}\int_{\partial\Omega_{r}}\big\langle\nabla_{x}\varepsilon_{n, a}(x-\xi,b(t,\tau)),\nu(x_{0})\big\rangle
	\varphi(\xi,\tau)a(\tau)dS_{\xi} d\tau,
	\end{equation}
	and  the complementary part $J_{r}(x,t)$ is defined by
	\begin{equation}\label{jdelta}
	J_{r}(x,t):=\int_{0}^{t}\int_{\partial\Omega\setminus\partial\Omega_{r}}\big\langle\nabla_{x}\varepsilon_{n, a}(x-\xi,b(t,\tau)),\nu(x_{0})\big\rangle
	\varphi(\xi,\tau)a(\tau)dS_{\xi} d\tau.
	\end{equation}
	Also, we denote
	\begin{equation}\label{i'delta}
	I_{r}^{'}(x,t):=\int_{0}^{t}\int_{\partial\Omega_{r}^{'}}\big\langle\nabla_{x}\varepsilon_{n, a}(x-\xi^{'},b(t,\tau)),\nu(x_{0})\big\rangle
	\varphi(x_{0},\tau)a(\tau)dS_{\xi^{'}}d\tau,
	\end{equation}
	where $dS_{\xi^{'}}$ is the surface element (at $\xi^{'}$) on $T(x_{0})$.
	To prove \eqref{e3.1}, we show that
	\begin{equation}\label{e3.2}
	\lim\limits_{x\to x_{0}}I_{r}^{'}(x,t)=\frac{1}{2}\varphi(x_{0},t),
	\end{equation}
	\begin{equation}\label{e3.3}
	\lim\limits_{x\to x_{0}}J_{r}(x,t)=J_{r}(x_{0},t),
	\end{equation}
	and
	\begin{equation}\label{e3.4}
	\lim\limits_{x\to x_{0}}\big(I_{r}(x,t)-I_{r}^{'}(x,t)\big)=I_{r}(x_{0},t).
	\end{equation}
	
	{\em Proof of \eqref{e3.2}}. Let us introduce a new variable $\tau \leftrightarrow\rho=\frac{|x-\xi^{'}|^{2}}{4b(t,\tau)}$ in \eqref{i'delta}. The substitution gives an implicit function $\tau=\tau(\rho)$ with $\frac{|x-\xi^{'}|^{2}}{4a_{1}(t)}\leq\rho<\infty$.
	Then integrating with respect to $\tau$, we obtain
	\begin{equation*}
	I_{r}^{'}(x,t)=\int_{\partial\Omega_{r}^{'}}\frac{\big\langle x-\xi^{'}, \nu(x_{0})\big\rangle}{|x-\xi^{'}|^{n}}\psi(x,\xi^{'},t)dS_{\xi^{'}},
	\end{equation*}
	where
	\begin{equation*}
	\psi(x,\xi^{'},t):=-\frac{1}{2}\pi^{-\frac{n}{2}}\int_{\frac{|x-\xi^{'}|^{2}}{4a_{1}(t)}}^{\infty}\rho^{\frac{n}{2}-1}e^{-\rho}\varphi\big(x_{0},\tau(\rho)\big)d\rho.
	\end{equation*}
	So $\psi(x,\xi^{'},t)$ is a continuous function of $(x,\xi^{'})$ for all $x=x_{0}-l\nu(x_{0})$ with $l>0$ sufficienty small and for all $\xi^{'}\in \partial\Omega_{r}^{'}$. In particular, we have
	\begin{equation*}
	\lim\limits_{\stackrel{x\to x_{0}}{\xi^{'}\to x_{0}}}\psi(x,\xi^{'},t)=-\frac{1}{2}\pi^{-\frac{n}{2}}\varphi(x_{0},t)\int_{0}^{\infty}\rho^{\frac{n}{2}-1}e^{-\rho}d\rho.
	\end{equation*}
	Since the integral is a value of the Gamma function and is equal to $\Gamma\big(\frac{n}{2}\big)=\frac{2\pi^\frac{n}{2}}{\omega_{n}}$, where $\omega_{n}$ is the area of the unit hypersphere in $\mathbb{R}^{n}$, we obtain
	$$\psi(x_{0},x_{0},t)=-\frac{\varphi(x_{0},t)}{\omega_{n}}.$$
	
	We divide $\partial\Omega_{r}^{'}=\partial\Omega_{1r}^{'}\cup R_{r}$ into two parts such that the boundary $\partial\Omega_{1r}^{'}$  contains $x_{0}$ in its interior exactly as in \cite[Theorem 1, Section 5.2]{F4}. On $\partial\Omega_{1r}^{'}$ we change the variables $\xi^{''}=\frac{\xi^{'}-x}{|\xi^{'}-x|}$ and  denote the domain of variation of $\xi^{''}$ by $\partial\Omega_{1r}^{''}$ and the corresponding area element by $dS_{\xi^{''}}$. Since $\langle x-\xi^{'}, \nu(x_{0})\rangle=-|x-\xi^{'}|\cos(\xi^{'}-x,\nu(x_{0}))$ and $\cos(\xi^{'}-x,\nu(x_{0}))dS_{\xi^{'}}=|x-\xi^{'}|^{n-1}dS_{\xi^{''}}$, we have
	\begin{eqnarray}\label{e3.5}
	\begin{aligned}
	&I_{r}^{'}(x,t)=\frac{\varphi(x_{0},t)}{\omega_{n}}\int_{\partial\Omega_{1r}^{''}}dS_{\xi^{''}}\\
	&-\int_{\partial\Omega_{1r}^{''}}\big(\psi(x,\xi^{'},t)-\psi(x_{0},x_{0},t)\big)dS_{\xi^{''}}+R_{r}(x,t),
	\end{aligned}
	\end{eqnarray}
	where the rest part of $I_{r}^{'}(x,t)$ denoted by $R_{r}(x,t)$, that is, the $\xi^{'}$ integration in $R_{r}(x,t)$ is taken over the set $R_{r}$.
	Since $\psi(x,\xi^{'},t)$ is a continuous function of $(x,\xi^{'})$, the second integral on the right-hand side of \eqref{e3.5} can be arbitrarily small. In $R_{r}(x,t)$, it should be noted that $\langle x-\xi^{'}, \nu(x_{0})\rangle\to 0$ as $x \to x_{0}$ and that if $\xi^{'}\in R_{r} $, then $|x-\xi^{'}|$ is bounded away from zero, which implies that the term $R_{r}(x,t)$ tends to zero. For the first term on the right-hand side of \eqref{e3.5}, we see that the boundary $\partial\Omega_{1r}^{''}$ tends to a unit hemisphere, thus, the first term in \eqref{e3.5} tends to $\frac{\varphi(x_{0},t)}{2}$.
	That is, we have proved \eqref{e3.2}.
	
	{\em Proof of \eqref{e3.3}}. For the variable $\xi$ in \eqref{jdelta} the inequality $|x-\xi|\geq \dfrac{r}{2}>0$ holds for $|x-x_{0}|<\dfrac{r}{2}$. Hence, the integral is a continuous function, which implies \eqref{e3.3}.
	
	{\em  Proof of \eqref{e3.4}}. To prove \eqref{e3.4}, we take $r_{1}>0$ such that $r_{1}<r$ and write
	\begin{equation}\label{e3.6}
	\begin{aligned}
	&I_{r}(x,t)=I_{r_{1}}(x,t)+\overline{I}_{r_{1}}(x,t),\\
	&I_{r}(x_{0},t)=I_{r_{1}}(x_{0},t)+\overline{I}_{r_{1}}(x_{0},t),\\
	&I_{r}^{'}(x,t)=I_{r_{1}}^{'}(x,t)+\overline{I}_{r_{1}}^{'}(x,t),
	\end{aligned}
	\end{equation}
	where $\overline{I}_{r_{1}}(x,t)$ $\Big(\overline{I}_{r_{1}}^{'}(x,t)\Big)$ is the complementary part to $I_{r_{1}}(x,t)$ $\Big({I}_{r_{1}}^{'}(x,t)\Big)$, that is, the $\xi$ $(\xi^{'})$-integration is taken over $\partial\Omega_{r}\setminus\partial\Omega_{r_{1}}$ $\Big(\partial\Omega_{r}^{'}\setminus\partial\Omega_{r_{1}}^{'}\Big) $. Note that $I_{r}^{'}(x_{0},t)=0$, since $\nu(x_{0})\perp (x_{0}-\xi^{'})$.
	
	Equality \eqref{e3.4} will be proved by showing that  for any $\varepsilon>0$ there exists $r_{1}>0$ such that
	\begin{equation}\label{e3.7}
	|I_{r_{1}}(x,t)-I_{r_{1}}^{'}(x,t)|<\varepsilon,
	\end{equation}
	\begin{equation}\label{e3.8}
	\begin{aligned}
	&|\overline{I}_{r_{1}}(x,t)-\overline{I}_{r_{1}}(x_{0},t)|<\varepsilon,\\
	&\qquad\enspace|\overline{I}^{'}_{r_{1}}(x,t)|<\varepsilon,
	\end{aligned}
	\end{equation}
	and
	\begin{equation}\label{e3.9}
	|I_{r_{1}}(x_{0},t)|<\varepsilon.
	\end{equation}
		{\em  Proof of \eqref{e3.7}}. Note that 
	\begin{equation}\label{inq1}
	|\xi-\xi^{'}|\leq C|x_{0}-\xi|^{1+\lambda},
	\end{equation}
	\begin{equation}\label{inq2}
	0<C_{1}\leq \frac{|x-\xi|}{|x-\xi^{'}|}\leq C_{2},
	\end{equation}
	where $C,~C_{1},$ and $C_{2}$ are constants. For the proofs of \eqref{inq1} and \eqref{inq2} we refer \cite[p. 135]{F4}. Now by using \eqref{inq1} and \eqref{inq2}, we obtain
	\begin{equation}\label{inq3}
	\begin{aligned}
	&\big|\big\langle x-\xi, \nu(x_{0})\big\rangle-\big\langle x-\xi^{'}, \nu(x_{0})\big\rangle\big|= |\xi-\xi^{'}|\leq C|x_{0}-\xi|^{1+\lambda}\\
	& \leq C |x-\xi|^{1+\lambda}.
	\end{aligned}
	\end{equation}
	By the mean value theorem and \eqref{inq2}, we get
	\begin{equation}\label{inq4}
	\begin{aligned}
	&\bigg|e^{-\frac{|x-\xi|^2}{4b(t,\tau)}}-e^{-\frac{|x-\xi^{'}|^2}{4b(t,\tau)}}\bigg|\leq e^{-\frac{K|x-\xi|^2}{4b(t,\tau)}}\frac{\big||x-\xi|^2-|x-\xi^{'}|^2\big|}{4b(t,\tau)}\\
	&\leq C e^{-\frac{K|x-\xi|^2}{4b(t,\tau)}}\frac{|x-\xi|^2}{b(t,\tau)},
	\end{aligned}
	\end{equation}
	where $K$ and $C$ are positive constants. 
	Combining \eqref{inq3}, \eqref{inq4} we obtain
	\begin{equation}\label{inq7}
	\begin{aligned}
	&\bigg|\frac{\big\langle x-\xi,\nu(x_{0})\big\rangle a(\tau) }{\big[b(t,\tau)\big]^{1+\frac{n}{2}}}e^{-\frac{|x-\xi|^2}{4b(t,\tau)}}-\frac{\big\langle x-\xi^{'},\nu(x_{0})\big\rangle a(\tau) }{\big[b(t,\tau)\big]^{1+\frac{n}{2}}}e^{-\frac{|x-\xi^{'}|^2}{4b(t,\tau)}}\bigg|\\
	&\leq \frac{a(\tau) }{\big[b(t,\tau)\big]^{1+\frac{n}{2}}}e^{-\frac{|x-\xi|^2}{4b(t,\tau)}}\big|\langle x-\xi, \nu(x_{0})\rangle-\langle x-\xi^{'}, \nu(x_{0})\rangle\big|\\
	& +\frac{|x-\xi^{'}|\big|\cos(\nu(x_{0}),\xi^{'}-x)\big|a(\tau) }{\big[b(t,\tau)\big]^{1+\frac{n}{2}}}\bigg|e^{-\frac{|x-\xi|^2}{4b(t,\tau)}}-e^{-\frac{|x-\xi^{'}|^2}{4b(t,\tau)}}\bigg|\\
	&\leq C_{1}|x-\xi|^{1+\lambda}\frac{a(\tau) }{\big[b(t,\tau)\big]^{1+\frac{n}{2}}}e^{-\frac{|x-\xi|^2}{4b(t,\tau)}}\\
	&+C_{2}\frac{|x-\xi|^2}{b(t,\tau)}\frac{\big|x-\xi|a(\tau) }{\big[b(t,\tau)\big]^{1+\frac{n}{2}}}e^{-\frac{K|x-\xi|^2}{4b(t,\tau)}}.
	\end{aligned}
	\end{equation}
	Using \eqref{inq6} for the case $s=\frac{|x-\xi|^{2}}{4b(t,\tau)}$ $\beta=1+\frac{n}{2}-\gamma_{1}$  to the first term of the last estimate of \eqref{inq7} and using \eqref{inq6} for the case $s=\frac{K|x-\xi|^{2}}{4b(t,\tau)}$, $\beta=2+\frac{n}{2}-\gamma_{2}$ to the second term of the last estimate of \eqref{inq7}, we see that the last estimate of \eqref{inq7} is bounded by
	\begin{equation}\label{inq8}
	\frac{\widetilde{C}_{1}a(\tau)}{\big[b(t,\tau)\big]^{\gamma_{1}}|x-\xi|^{n+1-2\gamma_{1}-\lambda}}+\frac{\widetilde{C}_{2}a(\tau)}{\big[b(t,\tau)\big]^{\gamma_{2}}|x-\xi|^{n+1-2\gamma_{2}}},
	\end{equation}
	for $0<\gamma_{1}<1+\frac{n}{2}$ and $0<\gamma_{2}<2+\frac{n}{2}$, where $\widetilde{C}_{1}$ and $\widetilde{C}_{2}$ are positive constants. For $1-\frac{\lambda}{2}<\gamma<1$ we can choose $\gamma_{1}$ and $\gamma_{2}$ such that $\gamma_{1}=\gamma$ and $\gamma_{2}=\gamma+\frac{\lambda}{2}$, thus term \eqref{inq8} is bounded by
	\begin{equation}\label{inq9}
	\frac{\widetilde{C}a(\tau)}{\big[b(t,\tau)\big]^{\gamma}|x-\xi|^{n+1-2\gamma-\lambda}},
	\end{equation}
	where $\widetilde{C}=\max{\{\widetilde{C}_{1},\widetilde{C}_{2}\}}$. Hence, we have 
	\begin{equation}\label{e3.10}
	\begin{aligned}
	&|I_{r_{1}}(x,t)-I_{r_{1}}^{'}(x,t)|\leq \widetilde{C}\int_{0}^{t}\int_{\partial\Omega_{r_{1}}}\frac{a(\tau)}{\big[b(t,\tau)\big]^{\gamma}|x-\xi|^{n+1-2\gamma-\lambda}}dS_{\xi}d\tau \\
	&+\sup \bigg|\frac{\varphi(\xi,\tau)}{\cos(\nu(x_{0}), \nu(\xi))}-\varphi(x_{0},\tau)\bigg|\\
	&\times\Bigg|\int_{0}^{t}\int_{\partial\Omega_{r_{1}}^{'}}\frac{\partial\varepsilon_{n, a}(x-\xi^{'},b(t,\tau))}{\partial\nu(x_{0})}
	dS_{\xi^{'}}d\tau\Bigg|,
	\end{aligned}
	\end{equation}
	for some $1-\frac{\lambda}{2}<\gamma<1$ and $0<\lambda<1$.
	
	The integrand of the first term on the right-hand side of \eqref{e3.10} is integrable, thus we can choose $r_{1}$ small enough to  make the corresponding integral arbitrarily small. Since the second integral in \eqref{e3.10} coincides with $I_{r}^{'}$ when $r=r_{1}$ and $\varphi(x_{0},\tau)\equiv 1$, it is bounded independently of  $r_{1}$. Since $\varphi$ is a continuous function and $\cos\big( \nu(x_{0}), \nu(\xi)\big) \to 1$, the expression $\sup|\cdot|\to 0$ as $r_{1} \to 0$. This completes the proof of \eqref{e3.7}.
	
	{\em Proof of \eqref{e3.8}}. Since  $|x-\xi|$, $|x_{0}-\xi|$ and $|x-\xi^{'}|$ in $\overline{I}_{r_{1}}(x,t)$, $\overline{I}_{r_{1}}(x_{0},t)$ and $\overline{I}^{'}_{r_{1}}(x,t)$ are bounded away from zero, correspondingly, and $\cos(\nu(x_{0}),\xi^{'}-x)\to 0$ as $x\to x_{0}$,  for any fixed $r_{1}$, we have \eqref{e3.8}, if $x$ is close enough to $x_{0}$.
	
	{\em Proof of \eqref{e3.9}}. Estimates in \eqref{inq7}-\eqref{inq8} imply
	\begin{equation*}
	|I_{r_{1}}(x_{0},t)|\leq \int_{0}^{t}\int_{\partial\Omega_{r_{1}}}\frac{a(\tau)dS_{\xi} d\tau}{\big[b(t,\tau)\big]^{\gamma}|x_{0}-\xi|^{n+1-2\gamma-\lambda}},
	\end{equation*}
	for some $1-\frac{\lambda}{2}<\gamma<1$. So, we have \eqref{e3.9}, if $r_{1}$ is sufficiently small.
	
	As we have proved \eqref{e3.4}, combining together all the proofs, we arrive at
	\begin{equation*}
	\begin{aligned}
	&\lim\limits_{\stackrel{x\to x_{0}}{x\in K}}\big\langle \nabla_{x}(S\varphi)(x,t),\nu(x_{0})\big\rangle=\frac{1}{2}\varphi(x_{0},t)\\&+\int_{0}^{t}\int_{\partial\Omega}
	\frac{\partial\varepsilon_{n, a}(x_{0}-\xi,b(t,\tau))}{\partial\nu(x_{0}) }\varphi(\xi,\tau)a(\tau)dS_{\xi} d\tau.
	\end{aligned}
	\end{equation*}
\end{proof} 
Note that for any $x_{0}\in \partial\Omega$ and $t\in (0,T]$, the single layer potential satisfies the jump relation
\begin{equation}\label{jumpsingle}
\begin{aligned}
&\lim\limits_{\stackrel{x \to x_{0}}{x \in K^{'}}}\big\langle \nabla_{x}(S\varphi)(x,t),n(x_{0})\big\rangle=-\frac{1}{2}\varphi(x_{0},t)\\&+\int_{0}^{t}\int_{\partial\Omega}
\frac{\partial\varepsilon_{n, a}(x_{0}-\xi,b(t,\tau))}{\partial n(x_{0}) }\varphi(\xi,\tau)a(\tau)dS_{\xi} d\tau,
\end{aligned}
\end{equation}
where the limit is taken along the inward normal $n(x_{0})$ and $K^{'}:=K^{'}(x_{0})\subset\overline{\Omega}^{c}\cup\{x_{0}\}$. The proof of relation \eqref{jumpsingle} is simlilar to the one of Theorem \ref{thm3.2}.
Now we show the jump relation for the double layer potential for the degenerate parabolic equation \eqref{e1.1}.

\begin{thm}\label{thm3.3}
	The double layer potential \eqref{doublelayer} with the density $\varphi \in C(\partial  \Omega \times [0,T])$ can be continuously extended from $\Omega \times (0,T]$ to $\overline{\Omega} \times (0,T]$ with the limiting values
	\begin{equation*}
	\lim\limits_{x\to x_{0}}(D\varphi)(x,t)=-\frac{1}{2}\varphi(x_{0},t)+\int_{0}^{t}\int_{\partial\Omega}
	\frac{\partial\varepsilon_{n, a}(x_{0}-\xi,b(t,\tau))}{\partial\nu(\xi) }\varphi(\xi,\tau)a(\tau)dS_{\xi} d\tau,
	\end{equation*}	
	for $x_{0} \in \partial\Omega$ and $0<t\leq T$, where  the time integral exists as an improper integral and $\nu(\xi)$ is the outward normal.
\end{thm}
\begin{proof}[Proof of Theorem \ref{thm3.3}]
	For the proof  we use  the same technique in Theorem \ref{thm3.2}.
\end{proof}
Consider the operator $D:C\big(\partial\Omega\times[0,T]\big)\to C\big(\partial\Omega\times[0,T]\big)$ defined by $$(D\varphi)(x,t):=\int_{0}^{t}\int_{\partial\Omega}
\frac{\partial\varepsilon_{n, a}(x-\xi,b(t,\tau))}{\partial\nu(\xi) }\varphi(\xi,\tau)a(\tau)dS_{\xi} d\tau,$$
for $x \in \partial\Omega$ and $0<t<T$ with the improper time integral over $(0,T)$. Here $a(t)$ satisfies the assumption (a).
Now we introduce a new variable $z$ given by
$$z:=b(t,\tau).$$

This substitution gives an implicit function $\tau=\tau(z)$ and $\varepsilon_{n, a}(x-\xi,b(t,\tau))=\varepsilon_{n}(x-\xi,z)$. Then the operator $D$ can be written in the form

\begin{equation}\label{doublelayerH}
(D\varphi)(x,t)=\int_{0}^{a_{1}(t)}\int_{\partial\Omega}
\frac{\partial\varepsilon_{n}(x-\xi,z)}{\partial\nu(\xi) }\varphi(\xi,\tau(z))dS_{\xi} dz.
\end{equation}
By the equality
$$\big|\big\langle x-\xi,\nu(\xi)\big\rangle\big|\leq |x-\xi|^{1+\lambda},\quad x,\xi\in\partial\Omega\in C^{1+\lambda},$$ and  estimates in \eqref{inq7}-\eqref{inq8} we obtain the estimate
\begin{equation}\label{e3.11}
\bigg|\frac{\partial\varepsilon_{n}(x-\xi,z )}{\partial\nu(\xi)}\bigg|\leq \frac{M}{z^{\gamma}|x-\xi|^{n+1-2\gamma-\lambda}},~~z>0,~x\neq\xi,
\end{equation}
for all $0<\gamma<1+\frac{n}{2}$ and some constant $M>0$ which depends on $L$ and $\gamma$. From here if we choose $\gamma$ such that $1-\frac{\lambda}{2}<\gamma<1$, we see that the kernel of $D$ is weakly singular with respect to the integrals over $\Omega$ and over time.

From \eqref{e3.11} we see that $(D\varphi)(\cdot,0)=0$ in $\Omega$. Thus, $D\varphi$ is continuous in $\overline{\Omega\times (0,T)}$ only if $\varphi(\cdot,0)=0$ on $\partial\Omega$. Moreover, the density $\varphi$ can be continuously extended to $\partial\Omega\times (-\infty,T]$ by setting $\varphi(\cdot,t)=0$ for $t<0$ in $\Omega$. Then, from Theorem \ref{thm3.3} we see that the double layer potential is a solution of the homogeneous equation 
\begin{equation}\label{e3.12}
\lozenge_{a}u = 0,\quad\textrm{in}~\Omega\times(0,T),
\end{equation}
with  the initial condition
\begin{equation}\label{e3.13}
u(\cdot,0)=0, \quad\textrm{in}~\Omega,
\end{equation}
and the boundary condition
\begin{equation}\label{e3.14}
u=g, \quad \textrm{on}~\partial\Omega\times(0,T),  
\end{equation}    
provided the continuous density $\varphi$ solves the following boundary integral equation
\begin{equation}\label{e3.15}
\Big(-\frac{1}{2}I+D\Big)(\varphi)=g,
\end{equation}
where $I$ is the identity operator and $g$ satisfies the compatibility condition
\begin{equation}\label{e3.16}
g(\cdot,0)=0, \quad \textrm{on}~\partial\Omega.
\end{equation}
Here we assumed that $g$ satisfies condition \eqref{e3.16} to have solvability of  problem \eqref{e3.12}-\eqref{e3.14}.
The following theorems and corollary are valid.
\begin{thm} \label{thm3.4}
	The double layer operator $D:C\big(\partial\Omega\times[0,T]\big)\to C\big(\partial\Omega\times[0,T]\big)$ is compact.
\end{thm}
\begin{proof}[Proof of Theorem \ref{thm3.4}]
	Since the kernel of $D$ is weakly singular, we apply [\cite{K3}, Theorem 2.29 and Theorem 2.30] to complete the proof.
\end{proof}
\begin{cor} \label{coro3.5}
	The double layer potential \eqref{doublelayer} is continuous in $\overline{\Omega\times (0,T)} $ provided that the continuous density $\varphi$ satisfies the condition $\varphi(\cdot,0)=0$ on $\partial\Omega$.
\end{cor}
\begin{thm}\label{thm3.6}
	The double layer potential \eqref{doublelayer} is a solution of  the initial boundary value problem	\eqref{e3.12}-\eqref{e3.14} provided $\varphi \in C(\partial\Omega\times[0,T])$ solves the boundary integral equation \eqref{e3.15} for all $x \in\partial\Omega$ and $t\in(0,T)$.	
\end{thm}
\begin{proof}[Proof of Theorem \ref{thm3.6}]
	This follows from Theorem \ref{thm3.3} and Corollary \ref{coro3.5}. The compatibility condition  for $g$ \eqref{e3.16} ensures the identity $\varphi(\cdot,0)=0$ on $\partial\Omega$ for solutions to \eqref{e3.15}.
\end{proof}
Since the integral operator $D$ is compact, equation \eqref{e3.15} is solvable for each $g \in C\big(\partial\Omega\times[0,T]\big)$  by the Riesz theory \cite[Corollary 3.5]{K3}, if the homogeneous equation $\big(-\frac{1}{2}I+D\big)(\varphi)=0$ has only the solution $\varphi=0$ on $\partial\Omega\times[0,T]$. We note that  estimate \eqref{e3.11} is equivalent to
\begin{equation}\label{est1}
\bigg|\frac{\partial\varepsilon_{n,a}(x-\xi,b(t,\tau) )}{\partial\nu(\xi)}a(\tau)\bigg|\leq \frac{Ma(\tau)}{\big[b(t,\tau)\big]^{\gamma}|x-\xi|^{n+1-2\gamma-\lambda}}.
\end{equation}
From \eqref{est1} we have the estimate

\begin{equation}\label{e3.17}
\|(D\varphi)(\cdot,t)\|_{L^{\infty}{(\partial\Omega)}}\leq C\int_{0}^{t}\frac{a(\tau)}{\big[b(t,\tau)\big]^{\gamma}}\|\varphi(\cdot,\tau)\|_{L^{\infty}{(\partial\Omega)}}d\tau.
\end{equation}
for all $t \in (0,T]$ and some constant $C>0$  which depends on $\partial\Omega$ and $\gamma$.

Repeating this argument (by induction), we obtain
\begin{equation*}
\|(D^{k}\varphi)(\cdot,t)\|_{L^{\infty}{(\partial\Omega)}}\leq C^{k}B^{k-1}\int_{0}^{t}\frac{a(\tau)}{\big[b(t,\tau)\big]^{k(\gamma-1)+1}}\|\varphi(\cdot,\tau)\|_{L^{\infty}{(\partial\Omega)}}d\tau,
\end{equation*}
for all $k \in \mathbb{N}$ and $t\in (0,T]$, where
\begin{equation*}
B:=\int_{0}^{1}\frac{ds}{[s(1-s)]^{\gamma}}.
\end{equation*}
Hence,  there exists an integer $k_{0}$ such that

\begin{equation}\label{e3.20}
\|(D^{k_{0}}\varphi)(\cdot,t)\|_{L^{\infty}{(\partial\Omega)}}\leq Q\int_{0}^{t}\|\varphi(\cdot,\tau)\|_{L^{\infty}{(\partial\Omega)}}d\tau,
\end{equation}
for all $t\in [0,T]$ and some constant $Q>0$.

Let $\varphi$ be a solution of the equation $\big(-\frac{1}{2}I+D\big)(\varphi)=0$. Then, by iteration, we see that $\varphi$ solves 
$$\left(-\frac{1}{2}I+D^{k_{0}}\right)(\varphi)=0.$$ 
So, \eqref{e3.20} implies the following estimate
$$\|\varphi(\cdot,\tau)\|_{L^{\infty}{(\partial\Omega)}}\leq\|\varphi\|_{L^{\infty}}\frac{Q^{n}t^{n}}{n!},\quad t\in[0,T],$$
for all $n\in \mathbb{N}$. Hence, $\varphi=0$ on $\partial\Omega\times[0,T]$. Thus, we have proved the existence and uniqueness theorem.
\begin{thm}\label{thm3.7}
	The initial boundary problem \eqref{e3.12}-\eqref{e3.14} with boundary datum in $C\big(\partial\Omega\times[0,T]\big)$ satisfying the condition \eqref{e3.16} has a unique solution $u \in C\big(\overline{\Omega\times(0,T)}\big)$ and the solution can be given as the double layer potential with the density $\varphi$, where $\varphi$ satisfies integral equation \eqref{e3.15}.
\end{thm}
\begin{cor}\label{crl3.8}
	The initial boundary problem \eqref{e3.12}-\eqref{e3.14} with the homogeneous boundary condition has only the trivial solution.
\end{cor}

\section{Trace formulae}\label{sec4}
\begin{thm}\label{thm4.1}
	Let $a(t)$ satisfy the assumption (a). Then
	for any $f(x,t)\in C\big(\overline{\Omega\times(0,T)}\big)$, $\operatorname{supp}f(\cdot,t)\subset \Omega$, $t\in [0,T]$, 
	the volume potential with the density $f$ \eqref{volume}  satisfies the boundary condition
	\begin{eqnarray}\label{bcofV}
	\begin{aligned}
	-\frac{u(x,t)}{2}+\int_{0}^{t}\int_{\partial\Omega}\frac{\partial\varepsilon_{n, a}(x-\xi,b(t,\tau))}{\partial\nu(\xi) }a(\tau)u(\xi,\tau)dS_{\xi}d\tau\\
		-\int_{0}^{t}\int_{\partial\Omega}\varepsilon_{n, a}(x-\xi,b(t,\tau))a(\tau)\frac{\partial u(\xi,\tau )}{\partial\nu(\xi) }dS_{\xi}d\tau=0,
		\end{aligned}
	\end{eqnarray}
	for all  $(x,t)\in \partial\Omega\times(0,T)$.
	
	Conversely, if $u(x,t)\in C_{x,t}^{2,1}(\Omega\times(0,T))\cap C(\Omega\times[0,T))\cap C_{x,t}^{1,0}(\overline{\Omega}\times(0,T)) $ is a solution of the equation 
	\begin{equation}\label{equV}
	\lozenge_{a}u = f\quad \text{in}~\Omega\times(0,T),
	\end{equation}
	satisfying the initial condition 
	\begin{equation}\label{conV}
	u(\cdot,0) = 0\quad \text{in}~\Omega,
	\end{equation}
	 and the lateral boundary condition \eqref{bcofV} for all $x\in \partial\Omega,~t \in (0,T)$, then it has a  unique solution $u$, which coincides with the volume potential.
\end{thm}

Note that the one-dimensional version of this theorem was stated in \cite{SO}.

\begin{proof}[Proof of Theorem \ref{thm4.1}]
 Since $f(x,t)\in C\big(\overline{\Omega\times(0,T)}\big)$ with  $\operatorname{supp}f(\cdot,t)\subset \Omega$ for all $t\in [0,T]$, the volume potential $(Vf)(x,t)$ solves Cauchy problem  \eqref{equV}-\eqref{conV}  by Theorem \ref{thm2.2}. Consider the volume potential
	\begin{align*}
	&(Vf)(x,t)=\int_{0}^{t}\int_{\Omega}
	\varepsilon_{n, a}(x-\xi,b(t,\tau))f(\xi,\tau)d\xi d\tau\\
	&=\int_{0}^{t}\int_{\Omega}\varepsilon_{n, a}(x-\xi,b(t,\tau))\bigg(\frac{\partial }{\partial \tau}-a(\tau)\Delta_{\xi}\bigg)(Vf)(\xi,\tau)d\xi d\tau.
	\end{align*}
	This integral is improper and defined as
	\begin{equation*}
	\lim\limits_{\delta \to 0}(V_{\delta}f)(x,t)=\lim\limits_{\delta \to 0}\int_{0}^{t-\delta}\int_{\Omega}
	\varepsilon_{n, a}(x-\xi,b(t,\tau))f(\xi,\tau)d\xi d\tau,
	\end{equation*}
	where $0<\delta<t$.
	
	We  use $\varepsilon_{n} \big(x-\xi,b(t,\tau)\big)$ in the sequel instead of $\varepsilon_{n,a}(x-\xi,b(t,\tau))$ for convenience. A direct calculation shows that
	\begin{equation}\label{e4.1}
	\begin{aligned}
	&(Vf)(x,t)=\lim\limits_{\delta \to 0} V_{\delta}(f)(x,t)\\
	&=\lim\limits_{\delta \to 0}\int_{0}^{t-\delta}\int_{\Omega}\varepsilon_{n, a}(x-\xi,b(t,\tau))\bigg(\frac{\partial }{\partial \tau}-a(\tau)\Delta_{\xi}\bigg)(Vf)(\xi,\tau)d\xi d\tau\\
	&=\lim\limits_{\delta \to 0}\int_{\Omega}\varepsilon_{n}\big(x-\xi,b(t,t-\delta)\big)(Vf)(\xi,t-\delta)d\xi \\
	&-\lim\limits_{\delta \to 0}\int_{\Omega}\varepsilon_{n}\big(x-\xi,b(t,0)\big)(Vf)(\xi,0)d\xi \\
	& +\lim_{\delta\to 0}\int_{0}^{t-\delta}
	\int_{\partial\Omega}\frac{\partial \varepsilon_{n}\big(x-\xi,
		b(t,\tau)\big)}{\partial\nu(\xi)}a(\tau)(Vf)(\xi,\tau)dS_{\xi} d\tau \\
	&-\lim_{\delta\to 0}\int_{0}^{t-\delta}
	\int_{\partial\Omega}\varepsilon_{n} 
	\big(x-\xi,b(t,\tau)\big)a(\tau)\frac{\partial (Vf)(\xi,\tau)}{\partial\nu(\xi)}dS_{\xi}d\tau \\
	&+\lim_{\delta\to 0}\int_{0}^{t-\delta}\int_{\Omega}(Vf)(\xi,\tau)
	\bigg(-\frac{\partial}{\partial\tau}-a(\tau)\Delta_{\xi}\bigg)\varepsilon_{n} 
	\big(x-\xi,b(t,\tau)\big)d\xi d\tau \\
	&=:I_{1}-I_{2}+I_{3}-I_{4}+I_{5},
	\end{aligned}
	\end{equation}	
	with the trivial definitions of $I_{i},~i=1,...,5$. Using the property of the fundamental solution of the heat operator such that $\varepsilon_{n}(x-\xi,t)\to \delta(x-\xi) ~\mbox{ with }~ t \to 0+$, we have
	\begin{align*}
	I_{1}&:=\lim_{\delta\to 0}\int_{\Omega} \varepsilon_{n} \big(x-\xi,b(t,t-\delta)\big)(Vf)
	(\xi,t-\delta)d\xi\\
	&=\int_{\Omega}\delta(x-\xi)(Vf)
	(\xi,t)d\xi=(Vf)(x,t).
	\end{align*} 
	Taking into account that $(Vf)(\cdot,0)=0$ in $\Omega$ , we get 
	\begin{equation*}
	I_{2}:=\lim\limits_{\delta \to 0}\int_{\Omega}\varepsilon_{n}\big(x-\xi,b(t,0)\big)(Vf)(\xi,0)d\xi =0.
	\end{equation*} 
	The integrals $I_{3},~I_{4}$ have limits as $\delta \to 0$
	\begin{equation*}
	\begin{aligned}
	&I_{3}:=\lim_{\delta\to 0}\int_{0}^{t-\delta}
	\int_{\partial\Omega}\frac{\partial \varepsilon_{n}\big(x-\xi,
		b(t,\tau)\big)}{\partial\nu(\xi)}a(\tau) (Vf)(\xi,\tau)dS_{\xi}d\tau\\ 
	&=\int_{0}^{t}
	\int_{\partial\Omega}\frac{\partial \varepsilon_{n}\big(x-\xi,
		b(t,\tau)\big)}{\partial\nu(\xi)}a(\tau) (Vf)(\xi,\tau)dS_{\xi}d\tau , \\
	&I_{4}:=\lim_{\delta\to 0}\int_{0}^{t-\delta}
	\int_{\partial\Omega}\varepsilon_{n} 
	\big(x-\xi,b(t,\tau)\big)a(\tau) \frac{\partial (Vf)(\xi,\tau)}{\partial\nu(\xi)}dS_{\xi} d\tau\\
	&=\int_{0}^{t}
	\int_{\partial\Omega}\varepsilon_{n}\big(x-\xi,b(t,\tau)\big)a(\tau)\frac{\partial (Vf)(\xi,\tau)}{\partial\nu(\xi)}dS_{\xi}d\tau. 
	\end{aligned}
	\end{equation*}
	Obviously, we have  
	\begin{equation*}
	I_{5}:=\lim_{\delta\to 0}\int_{0}^{t-\delta}\int_{\Omega}(Vf)(\xi,\tau)
	\bigg(-\frac{\partial}{\partial\tau}-a(\tau)\Delta_{\xi}\bigg)\varepsilon_{n} 
	\big(x-\xi,b(t,\tau)\big)d\xi d\tau=0.
	\end{equation*} 
	Taking into account \eqref{e4.1}, for all $(x,t)\in \Omega\times (0,T)$, we obtain
	\begin{equation}\label{e4.2}
	\begin{aligned}
	&I_{Vf}(x,t):=I_{3}-I_{4}=\int_{0}^{t}
	\int_{\partial\Omega}\bigg( \frac{\partial \varepsilon_{n}\big(x-\xi,
		b(t,\tau)\big)}{\partial\nu(\xi)}a(\tau)(Vf)(\xi,\tau) \\
	&-\varepsilon_{n}\big(x-\xi,b(t,\tau)\big)a(\tau)\frac{\partial  (Vf)(\xi,\tau)}{\partial\nu(\xi)}\bigg) dS_{\xi}d\tau=0.	
	\end{aligned}
	\end{equation}
	Using the jump relation of the double layer potential to \eqref{e4.2} with $x \to \partial\Omega$, we  arrive at
	\begin{align*}
	&I_{Vf}(x,t)|_{(x,t)\in \partial\Omega\times(0,T)}=-\frac{(Vf)(x,t)}{2}+\int_{0}^{t} 
	\int_{\partial\Omega}\bigg( \frac{\partial \varepsilon_{n}\big(x-\xi,
		b(t,\tau)\big)}{\partial\nu(\xi)}a(\tau)(Vf)(\xi,\tau)\\
	&-\varepsilon_{n}\big(x-\xi,b(t,\tau)\big)a(\tau)\frac{\partial (Vf)(\xi,\tau)}{\partial\nu(\xi)}\bigg) dS_{\xi}d\tau=0.	
	\end{align*}

	Now we show that if $u_{1}\in C_{x,t}^{2,1}(\Omega\times(0,T))\cap C(\Omega\times[0,T))\cap C_{x,t}^{1,0}(\overline{\Omega}\times(0,T))$ is a solution of problem \eqref{equV}, \eqref{conV}, \eqref{bcofV},  then $u_{1}$ is represented by the volume potential with the density $f$. If not, then  we assume that  $Vf$ and $u_{1}$  are solutions of problem \eqref{equV}, \eqref{conV}, \eqref{bcofV}. Then their difference $\omega=Vf-u_{1}$ must satisfy
	\begin{equation}\label{e4.3}
	\begin{aligned}
	&\lozenge_{a}\omega=0 \quad \text{in}~ \Omega\times(0,T),\\
	&\omega(\cdot,0)=0 \quad \text{in}~ \Omega,
	\end{aligned}
	\end{equation}
	and the boundary condition
	\begin{equation}\label{e4.4}
	\begin{aligned}
	&-\frac{\omega(x,t)}{2}+\int_{0}^{t}\int_{\partial\Omega}\frac{\partial\varepsilon_{n, a}(x-\xi,b(t,\tau))}{\partial\nu(\xi) }a(\tau)\omega(\xi,\tau)dS_{\xi}d\tau\\
	&-\int_{0}^{t}\int_{\partial\Omega}\varepsilon_{n, a}(x-\xi,b(t,\tau))a(\tau)\frac{\partial \omega(\xi,\tau )}{\partial\nu(\xi) }dS_{\xi}d\tau=0,
	\end{aligned}
	\end{equation}
	for all $(x,t)\in \partial\Omega\times(0,T)$. Since $f= 0$ in $\Omega\times(0,T)$, the representation formula \eqref{e4.1} has the form
	\begin{equation}\label{e4.5}
	\begin{aligned}
	&-\omega(x,t)=\int_{0}^{t} 
	\int_{\partial\Omega}\bigg( \frac{\partial \varepsilon_{n}\big(x-\xi,
		b(t,\tau)\big)}{\partial\nu(\xi)}a(\tau)\omega(\xi,\tau) \\
	&-\varepsilon_{n}\big(x-\xi,b(t,\tau)\big)a(\tau)\frac{\partial \omega(\xi,\tau)}{\partial\nu(\xi)}\bigg) dS_{\xi}d\tau,	
	\end{aligned}
	\end{equation}
	for all $(x,t)\in \partial\Omega\times(0,T)$. Using the jump relation of the double layer potential to \eqref{e4.5} with $x\to \partial\Omega$, we obtain
	\begin{equation}\label{e4.6}
	\begin{aligned}
	&-\omega(x,t)=-\frac{\omega(x,t)}{2}+\int_{0}^{t}
	\int_{\partial\Omega}\bigg( \frac{\partial \varepsilon_{n}\big(x-\xi,
		b(t,\tau)\big)}{\partial\nu(\xi)}a(\tau) \omega(\xi,\tau) \\
	&-\varepsilon_{n}\big(x-\xi,b(t,\tau)\big)a(\tau) \frac{\partial \omega(\xi,\tau)}{\partial\nu(\xi)}\bigg)dS_{\xi}d\tau,
	\end{aligned}
	\end{equation}
	where $(x,t)\in\partial\Omega\times (0,T)$.

	Comparing \eqref{e4.6} with \eqref{e4.4}, we conclude that 
	\begin{equation}\label{e4.7}
	\omega=0 \quad \mbox{  on }{\tiny }\partial\Omega\times(0,T).
	\end{equation}
	Corollary  \ref{crl3.8}  provides the existence of the unique trivial solution of the initial boundary value problem \eqref{e4.3}, \eqref{e4.7}. We have proved its uniqueness.	
\end{proof}
As in \cite{K8} we state the following theorem for the Poisson integral.
\begin{thm}\label{thm4.2}
	Let $a(t)$ satisfy the assumption (a) and $\varphi\in C(\Omega)$ with $\operatorname{supp}(\varphi)\subset \Omega$.
	Then the Poisson integral with the density $\varphi$ \eqref{Poisson} satisfies the lateral boundary condition
	\begin{eqnarray}\label{bcofS}
	\begin{aligned}
	 -\frac{u(x,t)}{2}+\int_{0}^{t}\int_{\partial\Omega}\frac{\partial\varepsilon_{n, a}(x-\xi,b(t,\tau))}{\partial\nu(\xi) }a(\tau)u(\xi,\tau)dS_{\xi}d\tau \\
		-\int_{0}^{t}\int_{\partial\Omega}\varepsilon_{n, a}(x-\xi,b(t,\tau))a(\tau)\frac{\partial u(\xi,\tau )}{\partial\nu(\xi) }dS_{\xi}d\tau=0,
		\end{aligned}
		\end{eqnarray}
	for all $x\in \partial\Omega,~t \in (0,T)$.
	
	Conversely, if $u(x,t)\in C_{x,t}^{2,1}(\Omega\times(0,T))\cap C(\Omega\times[0,T))\cap C_{x,t}^{1,0}(\overline{\Omega}\times(0,T))$ is a solution of
	\begin{equation}\label{equP}
	\lozenge_{a}u=0\quad\text{in}~\Omega\times(0,T),
	\end{equation}
	satisfying the initial condition
	\begin{equation}\label{conP}
	u(\cdot,0)=\varphi \quad \text{in} ~\Omega,
	\end{equation}
	and the lateral boundary condition \eqref{bcofS}
	for all $x\in \partial\Omega,~t \in (0,T)$, then $u$ coincides with the Poisson integral.
\end{thm}
\begin{proof}[Proof of Theorem \ref{thm4.2}]
	We use $\varepsilon_{n}\big (x-\xi,b(t,\tau)\big)$ instead of $\varepsilon_{n,a}(x-\xi,t-\tau)$ for convenience.
	For $x\in \Omega$ and $0<\delta<t$, it is easy to check that
	\begin{equation} \label{e4.9}
	\begin{aligned}
	&0=\lim_{\delta\to 0}
	\int_{0}^{t-\delta}\int_{\Omega} \varepsilon_{n,a} 
	(x-\xi,t-\tau) \bigg(\frac{\partial}{\partial\tau}-a(\tau)\Delta_{\xi}\bigg) 
	(P\varphi)(\xi,\tau)d\xi d\tau \\
	&=\lim_{\delta\to 0}\int_{\Omega} \varepsilon_{n} \big(x-\xi,b(t,t-\delta)\big)(P\varphi)
	(\xi,t-\delta)d\xi \\
	&\quad -\lim_{\delta\to 0}\int_{\Omega} \varepsilon_{n}\big (x-\xi,b(t,0)\big)(P\varphi)
	(\xi,0)d\xi \\
	&\quad +\lim_{\delta\to 0}\int_{0}^{t-\delta} 
	\int_{\partial\Omega}\frac{\partial \varepsilon_{n}\big(x-\xi,
		b(t,\tau)\big)}{\partial\nu(\xi)}a(\tau)(P\varphi)(\xi,\tau)dS_{\xi}d\tau \\
	&\quad -\lim_{\delta\to 0}\int_{0}^{t-\delta} 
	\int_{\partial\Omega}\varepsilon_{n}\big
	(x-\xi,b(t,\tau)\big)a(\tau)\frac{\partial (P\varphi)(\xi,\tau)}{\partial\nu(\xi)}dS_{\xi}d\tau\\
	&\quad +\lim_{\delta\to 0}\int_{0}^{t-\delta}\int_{\Omega}(P\varphi)(\xi,\tau)
	\bigg(-\frac{\partial}{\partial\tau}-a(\tau)\Delta_{\xi}\bigg)\varepsilon_{n} \big
	(x-\xi,b(t,\tau)\big)d\xi d\tau \\
	&=:J_{1}-J_{2}+J_{3}-J_{4}+J_{5},
	\end{aligned}
	\end{equation}
	with the definitions of $J_{i},~i=1,...,5$.
	Using the property of the fundamental solution of the heat operator such that $\varepsilon_{n}(x-\xi,t)\to \delta(x-\xi) ~\mbox{ with }~ t \to 0+$, we have 
	\begin{align*}
	J_{1}&:=\lim_{\delta\to 0}\int_{\Omega} \varepsilon_{n}\big (x-\xi,b(t,t-\delta)\big)(P\varphi)
	(\xi,t-\delta)d\xi\\
	&=\int_{\Omega}\delta(x-\xi)(P\varphi)
	(\xi,t)d\xi=(P\varphi)(x,t).
	\end{align*} 
	Since $(P\varphi)(\cdot,0)=\varphi$ , we have 
	
	\begin{equation*}
	J_{2}:=\lim_{\delta\to 0}\int_{\Omega} \varepsilon_{n} \big(x-\xi,b(t,0)\big)(P\varphi)
	(\xi,0)d\xi=(P\varphi)(x,t).
	\end{equation*}
	The integrals $J_{3},~J_{4}$ have limits as $\delta \to 0$
	\begin{align*}
	&J_{3}:=\lim_{\delta\to 0}\int_{0}^{t-\delta} 
	\int_{\partial\Omega}\frac{\partial \varepsilon_{n}\big(x-\xi,
		b(t,\tau)\big)}{\partial\nu(\xi)}a(\tau)(P\varphi)(\xi,\tau)dS_{\xi}d\tau\\
	&=\int_{0}^{t} 
	\int_{\partial\Omega}\frac{\partial \varepsilon_{n}\big(x-\xi,
		b(t,\tau)\big)}{\partial\nu(\xi)}a(\tau)(P\varphi)(\xi,\tau)dS_{\xi} d\tau, \\
	&J_{4}:=\lim_{\delta\to 0}\int_{0}^{t-\delta}
	\int_{\partial\Omega}\varepsilon_{n}\big
	(x-\xi,b(t,\tau)\big)a(\tau)\frac{\partial (P\varphi)(\xi,\tau)}{\partial\nu(\xi)}dS_{\xi}d\tau\\
	&=\int_{0}^{t}
	\int_{\partial\Omega}\varepsilon_{n} \big
	(x-\xi,b(t,\tau)\big)a(\tau)\frac{\partial (P\varphi)(\xi,\tau)}{\partial\nu(\xi)}dS_{\xi}d\tau.
	\end{align*}
	Clearly, 
	\begin{equation*}
	J_{5}:=\lim_{\delta\to 0}\int_{0}^{t-\delta}\int_{\Omega}(P\varphi)(\xi,\tau)
	\bigg(-\frac{\partial}{\partial\tau}-a(\tau)\Delta_{\xi}\bigg)\varepsilon_{n} \big
	(x-\xi,b(t,\tau)\big)d\xi d\tau=0.
	\end{equation*}
	From \eqref{e4.9} we get 
	\begin{align*}
	&I_{P\varphi}(x,t):=J_{3}-J_{4}=\int_{0}^{t} 
	\int_{\partial\Omega}\bigg( \frac{\partial \varepsilon_{n}\big(x-\xi,
		b(t,\tau)\big)}{\partial\nu(\xi)}a(\tau)(P\varphi)(\xi,\tau) \\
	&-\varepsilon_{n}\big(x-\xi,b(t,\tau)\big)a(\tau)\frac{\partial (P\varphi)(\xi,\tau)}{\partial\nu(\xi)}\bigg) dS_{\xi}d\tau=0.
	\end{align*}
	When $x \to \partial\Omega$ applying Theorem \ref{thm3.3}, we derive
	
	\begin{align*}
	&I_{P\varphi}(x,t)|_{(x,t)\in \partial\Omega\times(0,T)}=-\frac{(P\varphi)(x,t)}{2}+\int_{0}^{t} 
	\int_{\partial\Omega}\bigg( \frac{\partial \varepsilon_{n}\big(x-\xi,
		b(t,\tau)\big)}{\partial\nu(\xi)}a(\tau)(P\varphi)(\xi,\tau) \\
	&-\varepsilon_{n}\big(x-\xi,b(t,\tau)\big)a(\tau)\frac{\partial (P\varphi)(\xi,\tau)}{\partial\nu(\xi)}\bigg) dS_{\xi}d\tau.	
	\end{align*}
	The rest of the proof is the same as in the case of the volume potential.
\end{proof}

\end{document}